\documentclass{amsart}

\usepackage{amsfonts,amstext,amssymb, mathtools, comment,amsmath,amsthm, graphicx}

\newcommand{\levy}{L\'{e}vy }
\newcommand{\p}{{\mathbb P}}
\newcommand{\e}{{\mathbb E}}
\newcommand{\D}{{\mathrm d}}
\newcommand{\ep}{{\epsilon}}

\renewcommand{\a}{\alpha}
\renewcommand{\b}{\beta}
\renewcommand{\aa}{{\bs \alpha}}
\newcommand{\bb}{{\bs \beta}}
\newcommand{\qq}{{\bs q}}

\newcommand{\ind}[1]{\mbox{\rm 1}_{\{#1\}}}

\newcommand{\matI}{\mathbb{I}}
\newcommand{\matO}{\mathbb{O}}

\newcommand{\1}{{\bs 1}}
\newcommand{\ii}{{\rm i}}

\newcommand{\bs}{\boldsymbol}

\newcommand{\pii}{{\bs \pi}}

\newcommand{\rF}{{\hat F}}

\newtheorem{theorem}{Theorem}[section]
\newtheorem{proposition}{Proposition}[section]
\newtheorem{lemma}{Lemma}[section]
\newtheorem{corollary}{Corollary}[section]
\newtheorem{remark}{Remark}[section]
\newtheorem{definition}{Definition}[section]

\begin{document}
\title[Splitting and time reversal]{Splitting and time reversal 
for Markov additive processes}
\author[J. Ivanovs]{Jevgenijs Ivanovs}
\email{jevgenijs.ivanovs@unil.ch}
\address{University of Lausanne}

\thanks{Financial support by the Swiss National Science Foundation Project 200020 143889 is gratefully acknowledged.}
\begin{abstract}
We consider a Markov additive process with a finite phase space and study its path decompositions at the times of extrema, first passage and last exit.
For these three families of times we establish splitting conditional on the phase, and provide various relations between the laws of post- and pre-splitting processes using time reversal. These results offer valuable insight into behaviour of the process, and while being structurally similar to the \levy process case, they demonstrate various new features. As an application we formulate the Wiener-Hopf factorization, where time is counted in each phase separately and killing of the process is phase dependent. 
Assuming no positive jumps, we find concise formulas for these factors, and also characterize the time of last exit from the negative half-line using three different approaches, which further demonstrates applicability of path decomposition theory.
\end{abstract}

\maketitle

\section{Introduction}
In the theory of Markov processes \emph{path decomposition} or \emph{splitting time} theorems have a long history, see e.g.~\cite{millar_survey,pitman_levy_systems} and references therein. For a nice Markov process $X_t,t\geq 0$ these results are concerned with a study of random times $T$ such that the post-$T$ process $X_{T+t},t\geq 0$ has a Markovian structure and is independent of the events before $T$ given $X_T$ (and sometimes also $X_{T-}$). 
The main examples of such $T$ are stopping times, the time of the infimum and the last exit time from $(-\infty,a]$, all of which fall in the category of so-called randomized coterminal times~\cite{millar_survey}. 
The first example is rather trivial, and the latter two have the same structure, i.e.~the time of the infimum can be seen as the last exit time from a random interval $(-\infty,\inf_t X_t]$. Thus it is not surprising that they lead to the same law for the post-$T$ process: the law of $X$ conditioned to stay above a certain level. 

For a \levy process the theory of path decomposition becomes especially compelling and rich with various relations, see~\cite{millar_zero_one,duquesne} and also~\cite[Ch.\ VII.4]{bertoin}, for the one-sided case. It provides a valuable insight in the behaviour of the process and proves to be useful in various computations. In particular, the celebrated Wiener-Hopf factorization is just an application of splitting at the infimum, see~\cite{millar_zero_one, pitman_fluct} or~\cite[Lem.\ VI.6]{bertoin}.
The aim of this work is to develop the corresponding theory for Markov additive processes (MAPs), which are the processes with stationary and independent increments given the current phase. These processes often appear in finance, queueing and risk theories~\cite{APQ}, and recently were found to be instrumental in the analysis of real self-similar Markov processes, see ~\cite{kyprianou_appendix}. Even though MAPs closely resemble \levy processes, their increments are not exchangeable, and hence one may expect certain difficulties as well as apriory non-obvious differences from the \levy case. 

\subsection{Overview of the results}
The results of this paper are formulated for defective MAPs,
where general \emph{phase-dependent killing} is allowed. This is important in applications since it permits tracking \emph{time spent in different phases} by way of joint Laplace transforms, see Section~\ref{sec:defective}.
In the following we provide a brief overview of the main results and point out the main differences/difficulties as compared to the case of a \levy process, see also Figure~\ref{fig:main}. 
\begin{itemize}
\item Splitting and conditioning:
\begin{itemize}
\item Section~\ref{sec:split_inf} presents \emph{splitting at the infimum} and defines the law $\p_i^\uparrow$ of the post-infimum process given phase~$i$. There may be a phase switch at the infimum, see Figure~\ref{fig:scenarios}, and it is crucial to condition on the correct phase. Moreover, it is convenient to choose this phase in a slightly different way for infimum and supremum.
\item Section~\ref{sec:positive} shows that $\p_i^\uparrow$ corresponds to the original process \emph{conditioned to stay positive}. For certain phases $i$ this process starts from a positive level and possibly different phase.
This initial distribution is addressed in Proposition~\ref{prop:init_law}.
\item Section~\ref{sec:split_last} shows that \emph{splitting} also holds \emph{at the last exit time} $\sigma_a$ from an interval $(-\infty,a]$ given the exit is continuous,
and then the law of the post-$\sigma_a$ process is $\p_i^\uparrow$. Thus post-infimum and post-$\sigma_a$ processes have the same law given that they start in the same phase~$i$ (continuous exit can be realized only in some phases).
\end{itemize}
\item Time reversal and equivalence of laws:
\begin{itemize}
\item Section~\ref{sec:rev_zeta} discusses \emph{time reversal at the life time} of the process. 
In general, the life time depends on the process and hence the time reversal identity is different from the classical formula concerning reversal at an independent time.
\item Section~\ref{sec:rev_inf} expresses the law of the process \emph{reversed at the infimum} via the dual process conditioned to stay non-positive, see Theorem~\ref{thm:rev_inf}. Importantly, there are new non-trivial constants in this identity.
\item Section~\ref{sec:rev_last} shows that the process \emph{reversed at the last exit} is closely related to the dual process considered up to the last exit (on the events of continuous exit), see Theorem~\ref{thm:rev_last}.
There is another set of constants which have a clear probabilistic interpretation.
\item Section~\ref{sec:rev_first} completes the list of reversal identities by showing that \emph{reversal at the first passage} (when continuous) results in the dual process conditioned to stay positive and considered up to the last exit, see Theorem~\ref{thm:rev_first}. Again there is a need for appropriate scaling.
\end{itemize}
\item \emph{Wiener-Hopf factorization}: The statement of the factorization  and its relation to the results in~\cite{kaspi,klusik_palm,kyprianou_appendix} can be found in Section~\ref{sec:WH}. These works consider MAPs killed at independent exponential times, which would be natural for \levy processes, but is not fully satisfactory for MAPs. Moreover, we allow for counting time in each phase separately. Finally, we provide some necessary corrections to the factorizations in the literature.
\item \emph{Spectrally negative MAPs}: Section~\ref{sec:spNeg} further develops the above results in the important case when the process does not have positive jumps. In particular, the Wiener-Hopf factors are given by compact explicit formulas, which are then compared to the results in~\cite{dieker_mandjes,kyprianou_palm,ivanovs_thesis}. 
\item Application: Section~\ref{sec:sn_last} studies the times spent in different phases up to the last exit~$\sigma:=\sigma_0$ from the negative half-line (together with the phase at $\sigma$). Our theory allows for three different approaches: (i) is based on splitting at~$\sigma$ and law equivalence of post-$\sigma$ and post-infimum processes, (ii) is based on the formula for time reversal at~$\sigma$, and (iii) is based on splitting at the infimum and the relation between the post-infimum process considered up to its last exit and the dual process reversed at its first passage, see Theorem~\ref{thm:rev_first}.
\end{itemize}

The above list shows that splitting and time reversal results for MAPs closely resemble analogous results for \levy processes, which should help in understanding the present theory and its application. Importantly, there are non-trivial differences in each of these results, and we made an attempt to clearly stress these points. 
Some parts of the proofs are rather standard and so we keep them very brief. Many results, however, cannot be obtained by a simple generalization of the arguments in the \levy case. This is especially true when it comes to time reversal, see e.g.~Proposition~\ref{prop:rev_last}.
It should also be noted that Nagasawa's reversal theory for Markov processes~\cite{nagasawa} does not provide easy to use tools in our case, and instead we employ probabilistic arguments.

The subject of this work is rather technical, and our main goal is to present the results in a clear and comprehensible way while treating the details with care. Thus most of the results are supplemented with comments and remarks. Moreover, it is often convenient to draw a picture such as Figure~\ref{fig:main}.
\begin{figure}[h!]
\centering
\includegraphics[width=.4\linewidth]{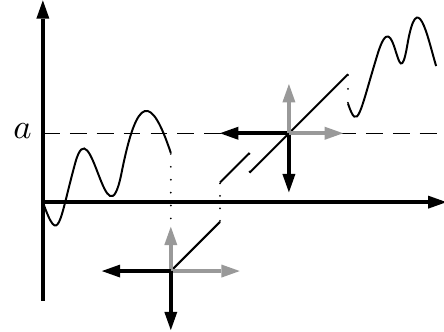}
\caption{Schematic sample path: splitting and time reversal.}
\label{fig:main}
\end{figure}
This picture illustrates splitting at the infimum, and the last exit $\sigma_a$ from $(-\infty,a]$ in a continuous way. Both post-splitting processes lead to the same law $\p_i^\uparrow$ corresponding to the process conditioned to stay positive, see grey axes. Time reversal at these times is depicted by small black axes - consider turning the paper upside down.

\section{Preliminaries and notation}\label{sec:prelim}
Consider a Markov additive process $(X_t,J_t),t\geq 0$, where $X$ is the additive component taking values in $\mathbb R$ and $J$ is the phase taking values in a finite set~$E$.
The defining property of a MAP reads: for all $T\geq 0$ given $\{J_T=i\}, i\in E$ the process $(X_{T+t}-X_T,J_{T+t}),t\geq 0$ is independent of $(X_t,J_t),t\in[0,T]$ and has the same law as the original process $(X,J)$ given $\{J_0=i\}$.
Moreover, we allow for killed processes by adding an absorbing state $(\infty,\dagger)$, and write $\zeta:=\inf\{t\geq 0:J_t=\dagger\}\in[0,\infty]$ to denote the life time of the process. We require the above defining property to hold, but we do \emph{not} add $\dagger$ to $E$.
Observe that $J_t$ is a (possibly transient) Markov chain on $E$ and assume that it is irreducible.
Finally, any MAP can be seen as a Markov-modulated \levy process: the additive component $X$ evolves as some \levy process $X^i$ while $J=i$ and has a jump distributed as some $U_{ij}$ at the phase switch from $i$ to $j$, where all these components and $J$ are independent. We refer to~\cite{APQ} and~\cite{ivanovs_thesis} for basic facts about MAPs.

It is common to write $\e[\cdot;J_t]$ for an $|E|\times|E|$ matrix with elements $\e_i(\cdot;J_t=j)=\e(\cdot\ind{J_t=j}|J_0=i)$.
It is well known that there exists a matrix-valued function $\Psi(\a)$ such that $\e[e^{\a X_t};J_t]=e^{\Psi(\a)t}$ for all $t\geq 0$ and at least purely imaginary~$\a$. It is noted that $\Psi(\a)$ is the analogue of the Laplace exponent of a \levy process, and that it characterizes the law of the corresponding MAP. Finally, we define first passage times:
\begin{align*}
&\tau_x:=\{t\geq 0: X_t>x\}, &\tau_x^-:=\{t>0: X_t<x\}.
\end{align*}

\subsection{Defective processes and time spent in different phases}\label{sec:defective}
Note that $\Psi(0)$ is the transition rate matrix of $J$, and hence $\qq:=-\Psi(0)\1$ is a vector with transition rates into $\dagger$ which are called the \emph{killing rates}.
Throughout this work we assume that our MAP is killed/defective:
\begin{align*}&\text{Assumption: } &\qq\neq \bs 0, \text{ i.e. }\zeta<\infty\text{ a.s.}\end{align*}
Note that $\zeta$ has a phase-type distribution (one under each $\p_i$) with phase generator $\Psi(0)$ and 
exit vector~$\qq$, see e.g.~\cite[Ch.\ III.4]{APQ}. Clearly, the life time~$\zeta$ depends on the process unless $q_i=q$ for all $i\in E$, in which case it is independent and exponentially distributed. 
On top of the implicit killing with rates $\qq$ we sometimes need additional killing with rates $\bb$; we write $\p^\bb$, $\Psi^\bb$ and so on, meaning that the underlying killing rates are $\bb+\qq$. Finally, some results of this paper (but not all) can be generalized to non-defective processes by taking $q_i\downarrow 0$ for all~$i$.

For any random time $T$ let $T_i=\int_0^T\ind{J_t=i}\D t$ be the time spent in phase~$i$ up to~$T$, and write $\bs T$ for the corresponding vector.
Note that by the properties of an exponential random variable we have
\[\e[e^{\a X_t-\langle\bb,\bs t\rangle};J_t]=\e^\bb [e^{\a X_t};J_t]=e^{\Psi^\bb(\a)t}.\]
Moreover, the joint transform of $X$ and $J$ at the killing time together with times spent in every phase is given by the following expression.
\begin{align}\label{eq:WH_easy}&\e[e^{\a X_{\zeta-}-\langle\bb,\bs \zeta\rangle};J_{\zeta-}]=\int_0^\infty\e[e^{\a X_t-\langle\bb,\bs t\rangle};J_t]\Delta_\qq\D t=\int_0^\infty e^{\Psi^\bb(\a)t}\D t\Delta_\qq\\&=-\left(\Psi^\bb(\a)\right)^{-1}\Delta_\qq,\nonumber\end{align}
where $\Delta_\qq:=(\ind{i=j}q_i)_{ij}$ denotes the diagonal matrix with $\qq$ on the diagonal. 
Here we used the fact that the set of jump points of $X$ has 0 Lebesgue measure, and that all the eigenvalues of $\Psi^\bb(\a)$ have negative real parts for $\a\in\ii\mathbb R,\bb\geq \bs 0$. It is noted that~\eqref{eq:WH_easy} should be compared to 
$\e(e^{\a X_{e_q}-\beta e_q})=-q/(\Psi(\a)-\b-q)$ in the \levy case, where $e_q$ is an independent exponential random variable with rate~$q$, see the left side of~(6.17) in~\cite{kyprianou}.

\subsection{Partition of phases}\label{sec:partition}
Recall that for a \levy process $X$ the point 0 is said to be regular for an open or closed set $B$ if $X$ hits $B$ immediately ($0-1$ event), and is irregular otherwise.
The conditions for regularity can be found in~\cite[Thm.\ 6.5]{kyprianou}.
In the following it will be convenient to partition the phases $E$ in three groups $E^\nearrow,E^\searrow,E^\sim$:
\begin{itemize}
\item $E^\nearrow$ contains $i\in E$ such that $0$ is regular for $(0,\infty)$ and irregular for $(-\infty,0]$ under $\p_i$; the prime example: $X^i$ has bounded variation and positive drift.
\item $E^\searrow$ contains $i\in E$ such that $0$ is irregular for $(0,\infty)$ and regular for $(-\infty,0]$ under $\p_i$; the prime example: $X^i$ has bounded variation and negative drift, also it can be a compound Poisson process.
\item $E^\sim$ contains $i\in E$ such that $0$ is regular for $(0,\infty)$ and for $(-\infty,0]$ under $\p_i$; the prime example: $X^i$ has unbounded variation.
\end{itemize}
It is noted that if $X^i$ has bounded variation and 0 drift then $i$ may belong to any of the three groups.

\section{Splitting and conditioning}
\subsection{Splitting at extrema}\label{sec:split_inf}
Define the overall infimum and its (last) time, and the overall supremum and its (first) time:
\begin{align*}&\underline X=\inf_{t\in[0,\zeta)}\{X_t\}, &\underline G=\sup\{t\in[0,\zeta):X_t\wedge X_{t-}=\underline X\},\\
&\overline X=\sup_{t\in[0,\zeta)}\{X_t\}, &\overline G=\inf\{t\in[0,\zeta):X_t\vee X_{t-}=\overline X\}.\end{align*}
Note that the distinction between first and last extrema is only necessary if for some $i$ the underlying \levy process $X^i$ is a compound Poisson process; in this case one may similarly consider first infimum and last supremum times.
\subsubsection{Splitting at the infimum}
\begin{proposition}\label{prop:splitting} The following splitting result holds true.
\begin{itemize}
\item[(i)] On the event $\{X_{\underline G}=\underline X\}$ the process $(X,J)$ splits at $\underline G$: given $J_{\underline G}$ the processes $(X_t,J_t),t\in[0,\underline G]$ and $(X_{\underline G+t}-\underline X,J_{\underline G+t}),t\geq 0$ are independent.
\item[(ii)] On the event $\{X_{\underline G}>\underline X\}$ the process $(X,J)$ splits at $\underline G-$: given $J_{\underline G-}$ the processes $(X_t,J_t),t\in[0,\underline G)$ and $(X_{\underline G+t}-\underline X,J_{\underline G+t}),t\geq 0$ are independent.
\end{itemize}
\end{proposition}
In the case of a \levy process the event $\underline X=X_{\underline G}$
has probability either~1 or~0, where the first corresponds to processes of types $E^\nearrow$ and $E^\sim$ and the second to type $E^\searrow$. Hence the splitting result can be formulated for these two cases separately, see \cite[Lem.\ VI.6]{bertoin}. This issue is more complicated for MAPs.
Letting
\begin{equation}\label{eq:underlineJ}\underline J=J_{\underline G}\ind{X_{\underline G}=\underline X}+J_{\underline G-}\ind{X_{\underline G}>\underline X},\end{equation}
be the phase in which the infimum was achieved, we have the following important observation.
\begin{lemma}\label{lem:underlineJ}
The following dichotomy holds with probability one:
\begin{itemize}
\item If $\underline J=i\in E^\sim\cup E^\nearrow$ then $X_{\underline G}=\underline X$ and $J_{\underline G}=i$,
\item If $\underline J=i\in E^\searrow$ then $X_{\underline G}>\underline X$ and $J_{\underline G-}=i$.
\end{itemize}
\end{lemma}
\begin{proof}
Consider separately the cases when the phase is switched and not switched at~$\underline G$, see also Remark~\ref{rem:switch}. Note also that $J_{\underline G}\in E^\sim\cup E^\nearrow$ does not imply that $X_{\underline G}=\underline X$ according to the first scenario in Remark~\ref{rem:switch}.
\end{proof}
It is clear from Lemma~\ref{lem:underlineJ} and Proposition~\ref{prop:splitting} that $\{\underline J=i\}$ determines the law of the post-infimum process, and there is no need to parameterize it by the phase $J_0$; we denote this law by~$\p^\uparrow_i$.
\begin{definition}
Let $(X,J)$ under $\p_i^\uparrow$ be distributed as $(X_{\underline G+t}-\underline X,J_{\underline G+t}),t\geq 0$ under $\p$ given $\{\underline J=i\}$.
\end{definition}
If for some $i$ and all $j$ we have $\p_j(\underline J=i)=0$ then $\p^\uparrow_i$ is left undefined.
Now Proposition~\ref{prop:splitting} states that given $\{\underline J=i\}$ the process splits at $\underline G$ or at $\underline G-$ according to $i\in E^\sim\cup E^\nearrow$ and $i\in E^\searrow$, and the post-infimum process has the law $\p^\uparrow_i$.

\begin{remark}\label{rem:switch}
It is important to understand the cases when $J$ may switch at $\underline G$, say from phase $i$ to phase~$j$.
For this to occur $X^i$ must be able to achieve its infimum at an exponential time or $X^j$ must be able to achieve its infimum at 0.
According to~\cite[Prop.\ 2.1]{millar_zero_one} there are three (non-exclusive) cases:
\begin{itemize}
\item $i\in E^\searrow$ and $\Psi_{ij}(0)\p(U_{ij}>0)>0$ for some $j\neq i$;
\item $j\in E^\nearrow$ and $\Psi_{ij}(0)\p(U_{ij}<0)>0$ for some $i\neq j$; 
\item $i\in E^\searrow, j\in E^\nearrow$ and $\Psi_{ij}(0)\p(U_{ij}=0)>0$ for some $j$,
\end{itemize}
see Figure~\ref{fig:scenarios}. Note that in the last scenario one may alternatively  split at $\underline G-$; consider the stopping times of phase switches. Finally, the time of killing $\zeta$ can be interpreted as the time of phase switch from some $i$ to $\dagger$, and if $\zeta=\underline G$ then necessarily $i\in E^\searrow$.
\begin{figure}[h!]
\centering
\includegraphics[width=.8\linewidth]{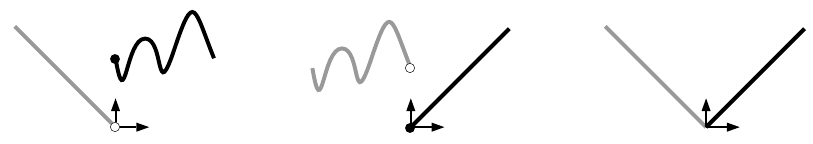}
\caption{Scenarios of phase switch at the infimum}
\label{fig:scenarios}
\end{figure}
\end{remark}


\begin{proof}[Proof of Proposition~\ref{prop:splitting}]
Splitting at the infimum for MAPs can be proven along the lines of~\cite{millar_zero_one} or~\cite{pitman_fluct}, and so we keep the proof rather brief.

For $\underline J=j\in E^\nearrow$ consider the epochs $T_n$ when $X$ is at the running minimum and $J$ is in~$j$. These epochs are stopping times, and hence we may apply classical splitting at each~$T_n$. Thus for any bounded functionals $f,g$ we have
 \begin{align*}&\e_i(f\{(X_t,J_t),t\in[0,\underline G]\}g\{(X_{\underline G+t}-\underline X,J_{\underline G+t}),t\geq 0\};\underline J=j)\\
&=\sum_n\e_i(f\{(X_t,J_t),t\in[0,T_n]\},T_n<\infty)\\
&\qquad\times\e_i(g\{(X_{T_n+t}-X_{T_n},J_{T_n+t}),t\geq 0\};X_{T_n+t}-X_{T_n}>0,t>0)\\
&=\sum_n\e_i(f\{(X_t,J_t),t\in[0,T_n]\};T_n<\infty,X_{T_n+t}-X_{T_n}>0,t>0)\\
&\qquad\times\e_j(g\{(X_t,J_t),t\geq 0\}|X_t>0,t>0)\\
&=\e_i(f\{(X_t,J_t),t\in[0,\underline G]\};\underline J=j)\e_j(g\{(X_t,J_t),t\geq 0\}|X_t>0,t>0),
\end{align*}
which establishes splitting at $\underline G$.

Next, consider the case $\underline J=j\in E^\sim$ which implies that there is no phase switch at~$\underline G$, and that there is no jump of $X$ at $\underline G$.
In this case we approximate $\underline G$ by an array of stopping times.
The standard way is to use the inverse local time at the infimum, see e.g.\ \cite[Lem.\ VI.6]{bertoin} for the \levy case and~\cite{klusik_palm,kyprianou_appendix} discussing the same procedure for MAPs. More precisely, let $L_t$ be a continuous local time process of $(X_t-\inf_{s\leq t}X_s,J_t)$ at $(0,j)$. On the event $n/m\leq L_\infty<(n+1)/m$ we apply classical splitting at the stopping time $L^{-1}_{n/m}$ similarly to the previous paragraph. Summing up over $n$ and taking $m\rightarrow\infty$ yields splitting at $\underline G$.

For $\underline J=j\in E^\searrow$ one may use the same argument as in the previous paragraph. Nevertheless, we present an alternative approach providing a better insight into the post-infimum process.
Consider the epochs of jumps of $X$ larger than $\epsilon>0$ which constitute a sequence of stopping times. 
Note that a given epoch $T_n$ coincides with $\underline G$ if $\inf_{t\geq 0}(X_{T_n+t}-X_{T_n-})>0$ and $\inf_{t<T_n} X_t=X_{T_n-}$. But for each of these epochs we can apply classical splitting right-before the jump, since the jump is also independent of the past given the corresponding phase. These ideas result in the following identity
\begin{align*}&\e_i(f\{(X_t,J_t),t\in[0,\underline G)\}g\{(X_{\underline G+t}-\underline X,J_{\underline G+t}),t\geq 0\};X_{\underline G}-X_{\underline G-}>\epsilon,J_{\underline G-}=j)\\
&=\e_i(f\{(X_t,J_t),t\in[0,\underline G)\};X_{\underline G}-X_{\underline G-}>\epsilon,J_{\underline G-}=j)\\
&\qquad\times\e_j(g\{(X^{(\epsilon)}_t,J^{(\epsilon)}_t),t\geq 0\}|\underline X^{(\epsilon)}>0),
\end{align*}
where $(X^{(\epsilon)},J^{(\epsilon)})$ has the law of the original process started from an independent $(X^{(\epsilon)}_0,J^{(\epsilon)}_0)$ and this latter pair under $\p_j$ denotes the first jump larger than $\epsilon$ while in phase~$j$, and the phase right-after this jump; see Section~\ref{sec:initial_distr} for further comments about the distribution of $(X^{(\epsilon)}_0,J^{(\epsilon)}_0)$. Finally, let $\epsilon\downarrow 0$ to deduce splitting at $\underline G-$.
\end{proof}
%


\subsubsection{Splitting at the supremum}
The result about splitting at the supremum can be obtained by taking $(-X,J)$ process, and adapting the statements and the proofs slightly, where it concerns a compound Poisson process $X^j$. Note also that when discussing splitting at supremum we assume that the absorbing state is $(-\infty,\dagger)$.
Define
\begin{equation}\label{eq:overlineJ}\overline J:=\ind{X_{\overline G-}=\overline X}J_{\overline G-}+\ind{X_{\overline G-}<\overline X}J_{\overline G},\end{equation}
where we take $\overline G-$ in the indicators as compared to $\underline G$ in the definition~\eqref{eq:underlineJ} of $\underline J$; by convention $0-:=0$. This choice turns out to be more convenient when it comes to time reversal at the infimum, see Section~\ref{sec:rev_inf}. In fact, the only difference appears when there is a phase switch at $\overline G$ but no jump of $X$, see the third scenario of Remark~\ref{rem:switch}: now we take the phase right before the switch and not right after.
  
Next, we note that on the set of probability one it holds that
\begin{itemize}
\item if $\overline J=i\in E^\sim\cup E^\nearrow$ then $X_{\overline G-}=\overline X$ and $J_{\overline G-}=i$,
\item if $\overline J=i\in E^\searrow$ then $X_{\overline G-}<\overline X$ and $J_{\overline G}=i$.
\end{itemize}
Finally, given $\{\overline J=i\}$ the process splits at $\overline G-$ or at $\overline G$ according to $i\in E^\sim\cup E^\nearrow$ and $i\in E^\searrow$, and the law of the post-supremum process is denoted by $\p^\downarrow_i$.

\subsection{Process conditioned to stay positive}\label{sec:positive}
For $x>0$ define $\p^{\uparrow}_{x,i}$ as the law of $(X,J)$ conditioned on $X_0=x,J_0=i$ and $X_t>0$ for all $t>0$. Clearly, $(X,J)$ under this law is a time-homogeneous Markov process with the semigroup
\begin{align}\label{eq:semigroup}\p^{\uparrow}_{x,i}(X_t\in\D y,J_t=j)&=\p_{x,i}(X_t\in\D y,J_t=j|\tau_0^-=\infty)\\&=\p_{x,i}(X_t\in\D y,J_t=j,t<\tau_0^-)\frac{\p_{y,j}(\tau_0^-=\infty)}{\p_{x,i}(\tau_0^-=\infty)}\nonumber\end{align}
for $x,y>0$. 
Recall that we only consider killed processes, and hence $\p_{x,i}(\tau_0^-=\infty)>0$.

\begin{remark}
There is a large body of literature devoted to \levy processes conditioned to stay positive, where the original process drifts to $-\infty$. This case is more complex than ours and may lead to different laws depending on the way conditioning (on the event of probability~0) is implemented, see~\cite{hirano}.
\end{remark}

The following result is a straightforward adaptation of~\cite[Prop.\ 4.1]{millar_zero_one}. It explains the notation chosen for the law of the post-infimum process. The corresponding result for \levy processes can also be found in~\cite{bertoin_splitting,chaumont_doney}.
\begin{proposition}
The post-infimum process, i.e.\ $(X,J)$ under $\p^\uparrow_i$, is a Markov process with transition semigroup defined in~\eqref{eq:semigroup} for $x>0$.
\end{proposition}

It is noted that if $i\in E^\nearrow\cup E^\sim$ then the post-infimum process starts at 0 and leaves 0 immediately, and hence the semigroup $\p_{x,j}^\uparrow,x>0,j\in E$ determines the law $\p^\uparrow_i$. Moreover, it can be shown that $\p^\uparrow_{x,i}$ converges on the Skorokhod space of paths to $\p^\uparrow_i$ as $x\downarrow 0$, see~\cite[Thm.\ 2]{chaumont_doney} for a similar statement concerning \levy processes. This result follows from splitting at the infimum and the fact that $\p_{x,i}(\underline G>\epsilon, \underline J\neq i|\underline X>0)\rightarrow 0$ which is obtained from the corresponding theory for \levy processes. Note also that for $i\in E^\nearrow$ the semigroup~\eqref{eq:semigroup} can be extended to include $x\geq 0$, because in this case $\p_{0,i}(\tau_0^-=\infty)>0.$ Then 
\[\p_i^\uparrow(\cdot)=\p_i(\cdot|\tau_0^-=\infty)=\p_i(\cdot|\underline X\geq 0),\]
which also follows from the proof of Proposition~\ref{prop:splitting}.

Finally, if $i\in E^\searrow$ then the corresponding result is less neat.
Compared to the \levy case there is an additional problem that 
$\p_{x,i}(\underline J\neq i|\underline X>0)$ does not necessarily converge to~$0$ (think of a model with $\Psi_{ij}(0)\p(U_{ij}=0)>0$).
Instead we use an alternative approach described in the following Section.

\subsubsection{On the initial distribution}\label{sec:initial_distr}
Let us focus on the law $\p_i^\uparrow$ for $i\in E^\searrow$, which corresponds to the post-infimum process starting from a positive level at time~0. 
The proof of Proposition~\ref{prop:split_sigma} implies that 
\[\e^\uparrow_ig\{(X_t,J_t),t\geq 0\}=
\lim_{\epsilon\downarrow 0}\e_i(g\{(X^{(\epsilon)}_t,J^{(\epsilon)}_t),t\geq 0\}|\underline X^{(\epsilon)}>0),\]
where $(X^{(\epsilon)},J^{(\epsilon)})$ has the law of the original process started from an independent $(X^{(\epsilon)}_0,J^{(\epsilon)}_0)$.
Notice that $\p_i(X^{(\epsilon)}_0\in\D x,J^{(\epsilon)}_0=j)$ is determined by competing independent exponential clocks: 
\begin{itemize}
\item for $j=i$ the rate is $\nu_i(\epsilon,\infty)$ and the level law is $\nu_i(\D x)/\nu_i(\epsilon,\infty)$,
\item for $j\in E\backslash\{i\}$ the rate is $\Psi(0)_{ij}\p(U_{ij}>\epsilon)$ and the level law is $\p(U_{ij}\in \D x)/\p(U_{ij}>\epsilon)$,
\item for $j=\dagger$ the rate is $q_i$ and the level equals $\infty$.
\end{itemize}
It is convenient to define
\begin{equation}\label{eq:Udef}U_{ij}(\D x):=\ind{i=j}\nu_i(\D x)+\ind{i\neq j}\Psi_{ij}(0)\p(U_{ij}\in \D x),\end{equation}
which we call the \emph{jump measure} associated to a MAP.
Now the above observations lead to the following result.

\begin{proposition}\label{prop:init_law}
For $i\in E^\searrow, q_i,x> 0$ it holds that
\begin{align*}\p_i^\uparrow( X_0\in \D x;J_0=j)&=\frac{1}{c_i+q_i}
v_j(x)U_{ij}(\D x),\\
\p_i^\uparrow(J_0=\dagger)&=
\frac{q_i}{c_i+q_i},
\end{align*}
where 
\begin{align*}&v_j(x):=\p_j(\underline X>-x),
 &c_i:=\int_0^\infty U_{i\cdot}(\D y)\bs v(y)\in[0,\infty).\end{align*}
\end{proposition}
\begin{proof}
Observe that
\begin{align*}
&\p_i^\uparrow(X_0> x;J_0=j)=\lim_{\epsilon\downarrow 0}\p_i(X^{(\ep)}_0> x;\underline X^{(\ep)}>0,J^{(\ep)}_0=j)/\p_i(\underline X^{(\ep)}>0)\\
&=\lim_{\epsilon\downarrow 0}\p_i( X^{(\ep)}_0> x;v_j(X^{(\ep)}_0),J^{(\ep)}_0=j)/(\sum_j\p_i(v_j(X^{(\ep)}_0);J_0^{(\ep)}=j)+\p_i(J_0^{(\ep)}=\dagger)),
\end{align*}
where $\p_i(X^{(\ep)}_0\in \D y, J^{(\ep)}_0=j)=U_{ij}(\D y)/r_i^{(\ep)}$ and $\p_i(J^{(\ep)}_0=\dagger)=q_i/r_i^{(\ep)}$ for $y>\ep$, and
$r_i^{(\ep)}$ is the total rate of the corresponding jumps.
Hence we have
\begin{equation}\label{eq:jump}\p_i^\uparrow( X_0> x;J_0=j)=\int_x^\infty v_j(y)U_{ij}(\D y)/\left(\int_0^\infty U_{i\cdot}(\D y)\bs v(y)+q_i\right),\end{equation}
where the first integral is clearly finite.
If $c_i=\infty$ then it must be that $\p_i^\uparrow(J_0=\dagger)=1$,
but the probabilistic reasoning shows that this is only possible when $U_{i\cdot}(0,\infty)\bs 1=0$, i.e.\ no jumps up in phase~$i$ are possible, showing that in fact $c_i=0$.
\end{proof}
Let us mention that we elaborate on this results a bit further in Section~\ref{sec:initial_sn} in the case of a spectrally positive processes.

\subsubsection{Process conditioned to stay non-positive}\label{sec:non_pos}
Similarly, the law $\p_i^\downarrow$ can be seen as the law of the process conditioned to stay non-positive. There are two features of $\p_i^\downarrow$ which are different from those of $\p_i^\uparrow$. Firstly, under $\p_i^\downarrow$ the process starts from the level~0 if $i\in E^\sim\cup E^\searrow$, but it may still do so for $i\in E^\nearrow$, see~\eqref{eq:overlineJ} and think about the exceptional scenario.
Secondly, if $X^i$ is a compound Poisson process then under $\p_i^\downarrow$ the process $X$ stays at the level~0 before jumping down.

\subsection{Splitting at the last exit from an interval}\label{sec:split_last}
Define the last exit time from~$(-\infty,a]$ by
\[\sigma_a:=\sup\{t\geq 0:X_t\leq a\}\]
and note that there are two trivial cases: $\sigma_a=0$ if $X_t>a$ for all $t\in[0,\zeta-)$, and $\sigma_a=\zeta$ if $X_{\zeta-}\leq a$.

The following result is well-known in the theory of Markov processes, see~\cite{millar_survey} and references therein.
\begin{proposition}\label{prop:split_sigma} 
Conditionally on $(X_{\sigma_a},J_{\sigma_a})$ the process 
$(X_{\sigma_a+t}-a,J_{\sigma_a+t}),t\geq 0$ is a Markov process with semigroup $\p_{x,j}^\uparrow,x>0,j\in E$ which is independent of $(X_t,J_t),t\in[0,\sigma_a).$
\end{proposition}
This result can be proven in a way similar to the proof of Proposition~\ref{prop:splitting}, that is, we make use of classical splitting at certain stopping times.

The main problem in relating the laws of the post-infimum and the post-$\sigma_a$ processes lies in the initial distribution of the latter. This problem is avoided if $X_{\sigma_a}=a$, i.e.\ the post-$\sigma_a$ process starts at~0. 
\begin{corollary}\label{cor:split_sigma}
On the event $\{X_{\sigma_a}=a,J_{\sigma_a}=i\}$ the process $(X_{\sigma_a+t}-a,J_{\sigma_a+t}),t\geq 0$ has the law $\p^\uparrow_{i}$, assuming that this event has positive probability, and in particular $i\in E^\nearrow\cup E^\sim$.
\end{corollary} 
\begin{proof}
It is easy to see that $i\notin E^\searrow$, because otherwise $X^i$ is either a compound Poisson process or a process which visits a level at discrete times (if ever) and goes immediately below this level. In any case the event of interest has 0 probability.
Now right-continuity of paths and convergence of $\p^\uparrow_{\epsilon,i}$ towards $\p^\uparrow_i$ as $\epsilon\downarrow 0$ implies the result.
\end{proof}
Note that the results of Proposition~\ref{prop:split_sigma} and Corollary~\ref{cor:split_sigma} also hold for the process $(X,J)$ under $\p^\uparrow$ with the same law for the post-$\sigma_a$ process, see also~\cite[Cor.\ VII.19]{bertoin} for the case of \levy processes with no positive jumps.

Finally, we identify the cases when $\{X_{\sigma_a}=a\}$ has positive probability.
It is said that a \levy process $X$ \emph{admits continuous passage upward} (creeps upward) if $\p(X_{\tau_x}=X_{\tau_x-})$ for some $x>0$ (and then for all), which is equivalent to $\p(X_{\tau_x}=x)>0$, assuming that $X$ is not a compound Poisson process. Sufficient and necessary conditions for this are given in~\cite[Thm.\ 7.11]{kyprianou}.

\begin{lemma}\label{lem:continuous_last}
Suppose $X$ is not monotone and $X^i$ is not a compound Poisson process. Then $\p_i(\sigma_a>0,X_{\sigma_a}=a,J_{\sigma_a}=j)>0$ if and only if $X^j$
admits continuous passage upward.
Moreover, on this event there is no jump and no phase switch at~$\sigma_a$.
\end{lemma}
\begin{proof}
Note that $X$ must be diffuse at each phase switch, and hence it can not be at the level $a$ at these times. Moreover, $X$ can not jump onto a level $a$. This proves the second statement.
Now the first statement follows from a similar result for \levy processes, see~\cite[Prop.\ 5.1]{millar_zero_one} based on time reversal argument.

Note that we exclude the case when $\sigma_0=0$, because one may take $X^i$ to be a process of bounded variation, zero drift and such that $i\in E^\nearrow$. Such a process does not admit continuous passage upward, but the above probability would be positive. A similar problem may arise for $\sigma_a>0$ if $X^i$ is a compound Poisson process whose jump measure has atoms.
\end{proof}



\section{Time reversal}\label{sec:timerev}
Throughout this section it will be convenient to work with the canonical probability space, where the sample space is a set of c\`adl\`ag paths $\omega$ equipped with the Skorokhod's topology and $(X_t(\omega),J_t(\omega))=\omega_t$.
Let us define the killing operator $k_T$ and the reversal operator $r_T$.
Namely,
\begin{align*}
k_T(\omega)_t&:=\ind{t<T}(X_t,J_t)+\ind{t\geq T}(\infty,\dagger),\\
r_T(\omega)_t&:=\ind{t<T}(X_{T}-X_{(T-t)-},J_{(T-t)-})+\ind{t\geq T}(\infty,\dagger),
\end{align*}
where $X,J$ and $T$ depend on~$\omega$. If $T\geq \zeta$ then we agree that $r_T$ produces a path identically equal to $(\infty,\dagger)$. Note that $r_T$ inverts both time and space, and that the additive component is reversed in the `additive' sense, see Figure~\ref{fig:main}. In addition, it leads to a c\`adl\`ag path which may start at some level $x\neq 0$ if there is a jump of $X$ at $T$; sometimes we will consider $r_{T-}$ in order to ignore this jump, i.e.\ when reversing at the life time~$\zeta$.
Next, for an arbitrary non-negative measurable functional $F$ we define
\begin{align*}&F_T:=F\circ k_T, &\qquad\rF_T:=F\circ r_T,\end{align*}
i.e\ we apply $F$ to the original process considered up to time~$T$ and to the process reversed at~$T$.  Finally, note that in general a `MAP killed at $T$' is not a (defective) MAP unless $T$ has a very particular law, see Section~\ref{sec:defective}.

Assume for a moment that $\qq=\bs 0$ and let $\pii$ be the stationary distribution of~$J$. Consider now a deterministic time~$t$ and note that neither $X$ nor $J$ jump at $t$ a.s. 
It is well-known (and easy to see) that if $J$ is in stationarity then the process time reversed at $t$ is also a MAP (with some law~$\hat \p$) killed at~$t$:
\[\e(\rF_t,J_0=i,J_t=j)=\hat\e(F_t,J_t=i,J_0=j).\]
Therefore, we have the following basic time reversal identity
\begin{equation}\label{eq:timerev1}
\pi_i\e_i(\rF_t,J_t=j)=\pi_j\hat\e_j(F_t,J_t=i),
\end{equation}
because $J$ is in stationarity under both $\p$ and $\hat \p$.
It is easy to see that this identity still holds true if the same $\qq\geq \bs 0$ is used to kill the process under $\e$ and $\hat \e$; notice that killing amounts to multiplying functionals $\rF_t$ and $F_t$ by $e^{-\langle \bs q,\bs t\rangle}$ which clearly preserves the correspondence, because $\bs t$ is the vector of total times spent in different phases and so it is invariant under time reversal.
Thus we again assume implicit killing with rates $\qq$ in the following, and determine $\pii$ by $\pii(\Psi(0)+\Delta_\qq)=\bs 0$.  

In matrix notation we have
\[\e[\rF_{t};J_{t}]=\Delta_\pii^{-1}\hat\e[F_{t};J_{t}]^T\Delta_\pii.\]
Taking $F=e^{\a X_{t-}}$ and noting that $F_t=\rF_t=e^{\a X_{t-}}$ we get 
\[e^{\Psi(\a)t}=\e[e^{\a X_t};J_t]=
\Delta_\pii^{-1}\hat\e[e^{\a X_t};J_t]^T\Delta_\pii=\Delta_\pii^{-1}e^{\hat \Psi(\a)^Tt}\Delta_\pii,\]
which implies a well-known relation
\[\hat \Psi(\a)=\Delta_\pii^{-1} \Psi(\a)^T\Delta_\pii.\]

In what follows we develop identities similar to~\eqref{eq:timerev1} for time reversal at the life time~$\zeta$, at the infimum $\underline G$, at the last exit time $\sigma_a$ and at the first passage time $\tau_x$ assuming the exit and passage are continuous.
It must be noted that there is a well-developed theory of time reversal for Markov processes, see~\cite{nagasawa,sharpe} and also~\cite{nagasawa_application} for an illustration. This theory of Nagasawa builds on potential theory, and thus it requires conversions between potential densities (resolvents) and the corresponding Markov processes, which makes it hard to apply in our case even when a resulting Markov process can be guessed by some other means. Instead, we follow a direct probabilistic path.

\subsection{Time reversal at the life time}\label{sec:rev_zeta}
Importantly, one can also time revert the process at the killing time $\zeta$ even though $\zeta$ depends on the process. In this case the relation becomes slightly different:
\begin{align}\label{eq:timerev2}&\pi_iq_i\e_i(\rF_{\zeta-};J_{\zeta-}=j)
=\pi_jq_j\hat\e_j(F_{\zeta-};J_{\zeta-}=i),\end{align}
which readily follows by integrating over life-time and noting that 
\[\pi_iq_i\int_0^\infty\e_i(\rF_{t};J_{t}=j)q_j\D t=\pi_jq_j\int_0^\infty\hat\e_j(F_{t};J_{t}=i)q_i\D t\]
according to~\eqref{eq:timerev1}.
Note that $F_{\zeta-}=F$, and above we chose the first one for the reason of symmetry only.
Note also, that killing at an independent exponential time $e_q$ means that $q_i=q$ and hence~\eqref{eq:timerev2} simplifies to~\eqref{eq:timerev1} with $t=\zeta-$, which is not true for general killing.

\subsection{Time reversal at the infimum}\label{sec:rev_inf}
Similarly to the case of \levy processes we may express the law of a MAP time reversed at its infimum  and supremum through the law of this MAP conditioned to stay non-positive and positive, respectively. These identities follow from time reversal at the killing time together with splitting at the extrema. Importantly, there is a new term in these identities as compared to the \levy case, see the definition of $\underline c_j$ below.

Recall that splitting at the infimum holds either at $\underline G$ or $\underline G-$ depending on the scenario, see also Figure~\ref{fig:scenarios}. This difficulty is resolved using the following definition:
\begin{align*}
\hat {\underline F}&:=\rF_{\underline G}\ind{X_{\underline G}= \underline X}+\rF_{\underline G-}\ind{X_{\underline G}>\underline X},
\end{align*}
which has to be compared with the definition of~$\underline J$.


\begin{theorem}\label{thm:rev_inf}
It holds that
\begin{align*}
\pi_iq_i\e_i(\hat{\underline F};\underline J=j)=c_j\hat\e^\downarrow_j(F_{\zeta-};J_{\zeta-}=i), 
\end{align*}
where $c_j:=\sum_{i}\pi_iq_i\p_i(\underline J=j)$.
\end{theorem}
\begin{proof} 
We may assume that the entries of $\qq$ are strictly positive, because the other case follows by taking limits.
Time reversal at the killing time~\eqref{eq:timerev2} and splitting at the infimum and also at the supremum for the reversed process yield
\begin{align*}
\pi_i q_i\e_i(\hat{\underline F};\underline J=j)\p_j^\uparrow(J_{\zeta-}=k)=
\pi_k q_k\hat \p_k(\overline J=j)\hat\e^\downarrow_j(F_{\zeta-};J_{\zeta-}=i).
\end{align*}
Here we have used the fact that $\underline J$ coincides with $\overline J\circ r_{\zeta-}$, see the corresponding definitions~\eqref{eq:underlineJ} and~\eqref{eq:overlineJ} and check the third scenario of Figure~\ref{fig:scenarios}.
In particular, we also have
\[\frac{\pi_iq_i\p_i(\underline J=j)}{\hat\p^\downarrow_j(J_{\zeta-}=i)}=\frac{\pi_k q_k\hat \p_k(\overline J=j)}{\p_j^{\uparrow}(J_{\zeta-}=k)}=:c_j,\]
which must be a constant depending on~$j$ only.
This proves the main identity and the expression for $c_j$ follows by taking $F=1$ and summing up over~$i$.
\end{proof}

We conclude this section by noting that Theorem~\ref{thm:rev_inf} will be useful in the context of the Wiener-Hopf factorization, see Section~\ref{sec:WH}.

\subsection{Time reversal at last exit}\label{sec:rev_last}
It is well-known that a \levy process time reversed at $\sigma_a$ has the law of the original process considered up to $\sigma_a$, where in both cases we condition on~$X_{\sigma_a}=a$. The aim of this section is to generalize this result to the setting of MAPs. It turns out that the MAP case is considerably harder to deal with, and that it leads to some additional non-trivial terms in the corresponding identity.
The following result lays the basis for time reversal at~$\sigma_a$. Moreover, it can be extended to provide an alternative proof of splitting at $\sigma_a$ (last exit becomes first passage for the process `run backwards').

\begin{proposition}\label{prop:rev_last}
If $X^j$ admits continuous passage upward then for any $a\in\mathbb R$ and any bounded continuous functional~$F$ it holds that 
\[\pi_iq_i\hat\e_i(\rF_{\sigma_a};X_{\sigma_a}=a,J_{\sigma_a}=j)=u_j L^a_{ji}(F),\]
where 
\begin{align}
\label{eq:u}u_j&:=\sum_k\int_0^\infty\pi_kq_k\p_k(X_{\tau_x}=x,J_{\tau_x}=j)\D x,\\
\label{eq:L}L^a_{ji}(F)&:=\lim_{\epsilon\downarrow 0}\frac{1}{\epsilon}\e_j(F;X_{\zeta-}\in(a,a+\epsilon),J_{\zeta-}=i).
\end{align}
\end{proposition}
\begin{proof}
Notice that $X^j$ hits points and so according to~\cite[Thm.\ II.16]{bertoin} the random variable $X^j(e_q)$ must have a density.
Standard results about convolutions show that the measure $\hat\p_i(J\text{ visits }j, X_{\zeta-}\in\D x)$ should be absolutely continuous (with respect to the Lebesgue measure), and so should be $\pi_iq_i\hat\e_i(\rF_{\sigma_a};X_{\sigma_a}=a,J_{\sigma_a}=j,X_{\zeta-}-a\in \D x,J_{\zeta-}=k)$.
 According to the Lebesgue differentiation theorem the corresponding density can be expressed as
\[\lim_{\epsilon\downarrow 0}\frac{\pi_iq_i}{\epsilon}\hat \e_i(\rF_{\sigma_a};X_{\sigma_a}=a,J_{\sigma_a}=j,X_{\zeta-}-a\in(x,x+\epsilon),J_{\zeta-}=k).\]
Now use time reversal at the killing time~\eqref{eq:timerev2} and splitting at~$\tau_x$ to express this density as
\[\pi_kq_k\p_k(X_{\tau_x}=x,J_{\tau_x}=j)L^a_{ji}(F),\]
see~\eqref{eq:L} and Figure~\ref{fig:rev_last}.
Integrating over $x>0$ and summing up over $k$ yields the result.

Application of time reversal and splitting requires additional commentary. 
Let $\tau=\sigma_a\circ_{\zeta-}$ be the time in the new coordinates corresponding to $\sigma_a$; in Figure~\ref{fig:rev_last} this is the first time from the right when the path hits the dotted line. Note that $X_{\tau}\in(x,x+\epsilon)$ and thus $\tau\in[\tau_{x},\tau_{x+\epsilon}]$. But $\tau_{x+\epsilon}$ converges to $\tau_x$, and so must the corresponding values of $(X,J)$. So in the limit we have $(X_{\tau_x},J_{\tau_x})=(x,j)$ on the event of interest.
Moreover, observe that $\p_k(x<X_{\tau_x}<x+\epsilon)\rightarrow 0$ and hence we may work on the event that $X$ upcrosses $x$ continuously in the pre-limit. Finally splitting at $\tau_{x}$ (in the pre-limit) yields the result.
\end{proof}
\begin{figure}[h!]
\centering
\includegraphics[width=.6\linewidth]{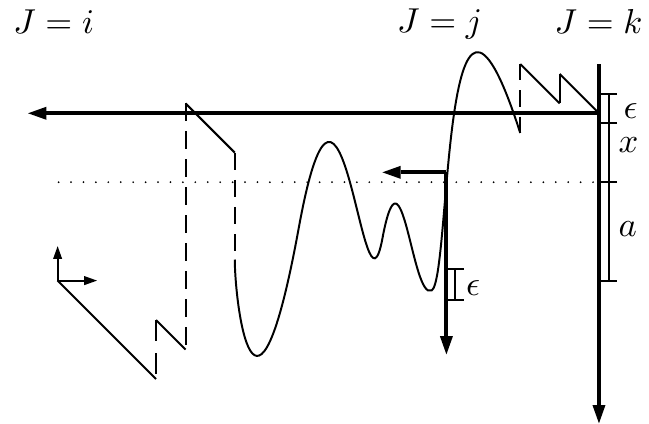}
\caption{Time reversal at $\zeta$ and splitting at $\tau_x$}
\label{fig:rev_last}
\end{figure}

\begin{remark}
It is noted that~\eqref{eq:L} equals
\[\lim_{\epsilon\downarrow 0}\frac{q_i}{\epsilon}\int_0^\infty\e_j(F_t;X_{t}\in(a,a+\epsilon),J_{t}=i)\D t,\]
and for $F=1$ this is just the potential density of $(X,J)$ at $(a,i)$ scaled by~$q_i$. The potential density, however, is determined for almost all levels~$a$ and hence one has to be careful.
In particular, one can substitute the interval $(a,a+\epsilon)$
by $(a-\epsilon/2,a+\epsilon/2)$ for almost all $a$, but not for $a=0$ when $i=j\in E^\nearrow$.
\end{remark}

We are ready to formulate our main result of this section.
\begin{theorem}\label{thm:rev_last}
For $a\in\mathbb R$ it holds that
\[u_i\e_i(\rF_{\sigma_a};X_{\sigma_a}=a,J_{\sigma_a}=j)=\hat u_j\hat \e_j(F_{\sigma_a};X_{\sigma_a}=a,J_{\sigma_a}=i),\]
where $u_i, \hat u_j$ are defined in~\eqref{eq:u}.
Moreover, both sides are identically zero if $X^i$ or $X^j$ does not admit continuous passage upward.
\end{theorem}
\begin{proof}
If $X^i$ does not admit continuous passage upward then $u_i=0$ and $\hat\p_j(X_{\sigma_a}=a,J_{\sigma_a}=i)=0$ apart from some exceptions. The latter can be violated only if $\sigma_a=0$ and so $i=j$ or $X^j$ is a compound Poisson process, see also Lemma~\ref{lem:continuous_last}. But both imply that $X^j$ does not admit continuous passage upward and so $\hat u_j=0$.

Hence in the following we can assume that both $X^i$ and $X^j$ admit continuous passage upward. Now Proposition~\ref{prop:rev_last} shows that
\begin{align*}
&\pi_iq_i\e_i(\rF_{\sigma_a};X_{\sigma_a}=a,J_{\sigma_a}=j)=\hat u_j\hat L^a_{ji}(F).
\end{align*}
A reasoning along the same lines also yields
\[\pi_jq_j\hat\e_j(F_{\sigma_a};X_{\sigma_a}=a,J_{\sigma_a}=i)=u_iL^a_{ij}(\rF_{\zeta-}).\]
Finally, time reversal at the killing time applied to~\eqref{eq:L} readily shows that 
\begin{equation}\label{eq:L_rel}\pi_iq_iL^a_{ij}(\rF_{\zeta-})=\pi_jq_j\hat L^a_{ji}(F),\end{equation}
which yields the claimed identity.
\end{proof}

\subsubsection{Alternative representation}
Let us note that there is a simpler version of Proposition~\ref{prop:rev_last} in the case when $j\in E^\nearrow$. Let $T_n^{a,j}$ be the successive epochs when $X_t=a,J_t=j$. Then by the strong Markov property we have
\begin{equation}\label{eq:rev_last_alt}\e_i(\rF_{\sigma_a};X_{\sigma_a}=a,J_{\sigma_a}=j)=\sum_n\e_i(\rF_{T_n^{a,j}};T_n^{a,j}<\infty)\p_j(\underline X>0).\end{equation}
Apart from being valid only for  $j\in E^\nearrow$ this version does not immediately yield the result of Theorem~\ref{thm:rev_last}, because it is not clear a-priory how to revert time at $T_n^{a,j}$. 
Nevertheless, there is a close link between~\eqref{eq:rev_last_alt} and Proposition~\ref{prop:rev_last}. 
Assume that $q_i,q_j>0$ and that $X^j$ admits continuous passage upward, i.e.\ it is a bounded variation process with drift $d_j>0$, see~\cite[Thm.\ 7.11]{kyprianou}. Then we have
\[\e_i(\rF_{\sigma_a};X_{\sigma_a}=a,J_{\sigma_a}=j)=\frac{1}{\pi_i q_i}\hat u_j\hat L^a_{ji}(F)=\frac{1}{\pi_jq_j}\hat u_jL^a_{ij}(\rF_{\zeta-}),\]
see~\eqref{eq:L_rel}.
But \[L^a_{ij}(\rF_{\zeta-})=\lim_{\epsilon\downarrow 0}\frac{q_j}{\epsilon}\int_0^\infty\e_i(\rF_t;X_{t}\in(a,a+\epsilon),J_{t}=j)\D t\]
and since $X^j_t/t\rightarrow d_j$ as $t\downarrow 0$, see~\cite[Prop.\ VI.11]{bertoin}, 
we conjecture that
\[L^a_{ij}(\rF_{\zeta-})=\frac{q_j}{d_j}\sum_n\e_i (\rF_{T_n^{a,j}};T_n^{a,j}<\infty).\]
It is not hard to prove this conjecture under assumption that the limit and the expectation can be interchanged.
Thus under this assumption the identity~\eqref{eq:rev_last_alt} and Proposition~\ref{prop:rev_last} are connected by the relation
\[\p_j(\underline X_t>0)=\frac{\hat u_j}{\pi_jd_j},\]
which is again rather clear intuitively.

\subsection{Time reversal at first passage}\label{sec:rev_first}
A \levy process reversed at $\tau_x,x>0$ on the event of continuous passage has the law of the process conditioned to stay positive and considered up to its last exit from $[0,x]$ on the event of continuous exit, see~\cite[Thm.\ 4.2]{duquesne}. The following Theorem shows that a similar result holds for MAPs. Here again we have additional terms $u_i,u_j$.
\begin{theorem}\label{thm:rev_first}
For $x>0$ it holds that 
\[u_i\e_i(\rF_{\tau_x},X_{\tau_x}=x,J_{\tau_x}=j)=u_j\hat\e_j^\uparrow(F_{\sigma_x};X_{\sigma_x}=x,J_{\sigma_x}=i).\]
Moreover, both sides are identically zero if $X^i$ or $X^j$ does not admit continuous passage upward.
\end{theorem}
\begin{proof}
The last claim is shown using similar arguments to those in the proof of Theorem~\ref{thm:rev_last} together with the interpretation of $\hat\e^\uparrow_j$ as the law of conditioned process when started from a positive level. In the following we assume that both $X^i$ and $X^j$ admit continuous passage upward.

Consider the process $(X,J)$ up to its last passage over $x$ on the event that $X_{\sigma_x}=x,J_{\sigma_x}=i$ and also $X_{\sigma_0}=0,J_{\sigma_0}=j$. Using time reversal at $\sigma_x$ and splitting at~$\sigma_0$, as well as at~$\tau_x$ for the reversed process, we write
\begin{align*}u_j&\p_j(X_{\sigma_0}=0,J_{\sigma_0}=j)\e_j^\uparrow(F_{\sigma_x};X_{\sigma_x}=x,J_{\sigma_x}=i)\\&=\hat u_i\hat\e_i(\rF_{\tau_x};X_{\tau_x}=x,J_{\tau_x}=j)\hat \p_j(X_{\sigma_0}=0,J_{\sigma_0}=j),\end{align*}
see Theorem~\ref{thm:rev_last} and Corollary~\ref{cor:split_sigma}.
Again using Theorem~\ref{thm:rev_last} we observe that
\[u_j\p_j(X_{\sigma_0}=0,J_{\sigma_0}=j)=\hat u_j\hat \p_j(X_{\sigma_0}=0,J_{\sigma_0}=j)\]
and hence
\[\hat u_j\e_j^\uparrow(F_{\sigma_x};X_{\sigma_x}=x,J_{\sigma_x}=i)=\hat u_i\hat\e_i(\rF_{\tau_x};X_{\tau_x}=x,J_{\tau_x}=j),\]
which readily leads to the result by taking the law $\hat\p$ instead of $\p$.
\end{proof}

\section{Wiener-Hopf factorization}\label{sec:WH}
The following result is a direct consequence of splitting and reversal at the infimum and supremum, see Proposition~\ref{prop:splitting} and Theorem~\ref{thm:rev_inf}. Recall also that $\bs t$ denotes a vector of total times spent in different phases up to time~$t$, see Section~\ref{sec:defective}.
\begin{corollary}[Wiener-Hopf]\label{cor:WH}
For $\a\in\ii\mathbb R,\bb\geq \bs 0$ and implicit killing rates $\qq\neq \bs 0$ it holds that
\begin{align*}\e[e^{\a X_{\zeta-}-\langle\bb, \bs \zeta\rangle};J_{\zeta-}]&=-\left(\Psi^\bb(\a)\right)^{-1}\Delta_\qq\\
&=\e[e^{\a \underline X-\langle\bb ,\underline {\bs G}\rangle};\underline J]\e^\uparrow[e^{\a X_{\zeta-}-\langle\bb,\bs \zeta\rangle};J_{\zeta-}]\\
&=\e[e^{\a \overline X-\langle\bb ,\overline {\bs G}\rangle};\overline J]\e^\downarrow[e^{\a X_{\zeta-}-\langle\bb,\bs \zeta\rangle};J_{\zeta-}].
\end{align*}
Moreover,
\begin{align}
\label{eq:WHup}\Delta_{\bs {\overline c}}\e^\uparrow[e^{\a X_{\zeta-}-\langle\bb,\bs \zeta\rangle};J_{\zeta-}]&=
\hat\e[e^{\a \overline {X}-\langle\bb,\bs {\overline {G}}\rangle};\overline {J}]^T\Delta_\pii\Delta_\qq,\\
\label{eq:WHdown}\Delta_{\bs {\underline c}}\e^\downarrow[e^{\a X_{\zeta-}-\langle\bb,\bs \zeta\rangle};J_{\zeta-}]&=
\hat\e[e^{\a \underline {X}-\langle\bb,\bs {\underline {G}}\rangle};\underline {J}]^T\Delta_\pii\Delta_\qq,
\end{align}
where the vectors $\bs {\overline c},\bs {\underline c}$ are given by
\begin{align}\label{eq:c}&\bs {\overline c}=\hat\p[\overline J]^T\Delta_\pii\qq, &\bs {\underline c}&=\hat \p[\underline J]^T\Delta_\pii\qq.\end{align}
\end{corollary}
\begin{proof}
The first identity follows immediately from~\eqref{eq:WH_easy} and Proposition~\ref{prop:splitting} together with its counterpart for splitting at the supremum.
The identity~\eqref{eq:WHdown} follows from Theorem~\ref{thm:rev_inf} by choosing 
$F=e^{\a X_{\zeta-}-\langle\bb,\bs \zeta\rangle}$. Similalrly~\eqref{eq:WHup} follows from reversal at the supremum.
\end{proof}

Importantly, there are new non-trivial terms $\Delta_{\bs {\overline c}}$ and $\Delta_{\bs {\underline c}}$ entering the Wiener-Hopf factorization, as compared to the \levy case. 
Observe that $\overline c_j=0$ if and only if $q_i\hat\p_i(\overline J=j)=0$ for all~$i$. Thus it is enough to assume that none of $X^j$ is monotone to guarantee that $\Delta_{\bs {\overline c}}$ is invertible. Under this assumption we have
\begin{equation}\label{eq:WH}\e[e^{\a X_{\zeta-}-\langle\bb, \bs \zeta\rangle};J_{\zeta-}]= 
\e[e^{\a \underline X-\langle\bb ,\underline {\bs G}\rangle};\underline J]\Delta_{\bs {\overline c}}^{-1}\hat\e[e^{\a \overline {X}-\langle\bb,\bs {\overline {G}}\rangle};\overline {J}]^T\Delta_\pii\Delta_\qq.\end{equation}
In fact, this identity is true in general if we interpret $0/0$ as $0$, or simply replace the 0 entries of $\bs{\overline c}$ by arbitrary non-zero numbers.
\subsection{Ladder processes}
In this Section we briefly discuss so-called ladder processes, see also~\cite{kaspi,kyprianou_appendix}. In our setting these ladder processes have  $1+|E|$ additive components, where one represents height and the others track time in different phases. The associated matrix exponents are then used to provide an alternative representation of the factors $\e[e^{\a \underline X-\langle\bb ,\underline {\bs G}\rangle};\underline J]$ and $\e[e^{\a \overline X-\langle\bb ,\overline {\bs G}\rangle};\overline J]$.

Let $L_t$ be a local time of $X_t-\inf_{s\leq t} X_s$ which can be constructed by considering the intervals of time between subsequent phase switches, where $X$ evolves as a \levy process (see e.g.~\cite[Ch.\ IV]{bertoin} for the corresponding theory). 
Hence we may choose $L$ so that it evolves continuously while in a phase belonging to~$E^\sim\cup E^\searrow$, and it is a jump process (with independent exponential jumps) when the phase is in~$E^\nearrow$.
In particular, there is an exponential jump of $\underline L$ at a phase switch~$T$ if and only if $X_T=\underline X$, see also Figure~\ref{fig:scenarios}. 

Let $L^{-1}_t$ be the inverse local time and as usual let $\bs L^{-1}_t$ denote the vector of total times spent in different phases up to $L^{-1}_t$, see Section~\ref{sec:defective}.
Observe that the ladder process $((\bs L^{-1}_t,X_{L^{-1}_t}),J_{L^{-1}_t})$ is a multivariate Markov additive process.
Hence there exists matrix valued function $\underline K(\aa,\b)$ such that
\[\e[e^{-\langle\aa,\bs L^{-1}_t\rangle+\b X_{L^{-1}_t}};J_{ L^{-1}_t}]=e^{\underline K(\aa,\b)t},\]
because $L^{-1}_0=0$; here $\a_i\geq 0$ and $\b\geq 0$.
Finally, observe that $((\bs {\underline G},\underline X),\underline J)$ has the distribution of the ladder process right-before killing time, see the definition of $\underline J$ in~\eqref{eq:underlineJ} and Figure~\ref{fig:scenarios}. 
Letting
\[\bs{\underline k}:=-\underline K(0,0)\bs 1\]
be the killing rates of the ladder process, we obtain
\begin{equation}\label{eq:K}\e[e^{\a \underline {X}-\langle\bb,\bs {\underline {G}}\rangle};\underline {J}]=\int_0^\infty e^{\underline K(\bb,\a)t}\D t\Delta_{\bs {\underline k}}=-\underline K(\bb,\a)^{-1}\Delta_{\bs {\underline k}},\end{equation}
because the Perron-Frobenius eigenvalue of $\underline K(\bb,\a)$ is negative.

In a very similar way we construct a ladder process corresponding to the supremum. Recall that there is a slight difference in the definition of $\overline J$, see also the third scenario in Figure~\ref{fig:scenarios}.
More precisely, there is a jump of the local time at a phase switch $T$ if and only if $X_{T-}<X_T=\overline X$; here we do not consider 0 as the time of phase switch. Irrespective of this difference we still have
\[\e[e^{\a \overline {X}-\langle\bb,\bs {\overline {G}}\rangle};\overline {J}]=-\overline K(\bb,\a)^{-1}\Delta_{\bs {\overline k}},\]
where $\a\leq 0,\b_i\geq 0$.

Finally, observe that the factorization in~\eqref{eq:WH} can be rewritten in terms of $\overline K(\bb,\a)$ and $\hat{\underline K}(\bb,\a)$. In this respect it may be useful to note that the local times can be arbitrarily scaled by positive constants, which may be different for different phases. This leads to scaling the rows of the associated matrix exponent. In particular, the scaling can be chosen such that all the entries of $\bs{\overline k}$ are either~0 or~1.


\subsection{Related literature and comments}
The Wiener-Hopf factorization for MAPs has appeared in the literature in various forms, see~\cite{kaspi,klusik_palm, kyprianou_appendix}.
These works assume that the killing rate does not depend on the phase, and count the total time only as opposed to counting time for each phase separately. Our result with a particular choice of $q_i=q>0$ and $\b_i=\b\geq 0$, see~\eqref{eq:WH} and Corollary~\ref{cor:WH}, closely reminds factorizations found in~\cite{klusik_palm} and~\cite{kyprianou_appendix}. There are, however, two substantial differences. 
Firstly, we use $\underline J$ and $\overline J$ instead of $J_{\underline G}$ and $J_{\overline G}$, and secondly, we have $\Delta_{\bs {\overline c}}^{-1}$ instead of $q^{-1}\Delta_\pii^{-1}$.
In the following we show that these changes are necessary for the factorization to hold in general.

The reason to take $\underline J$ should be clear from Section~\ref{sec:split_inf}: when the process jumps up from the infimum then the post-infimum process seen from $\underline G$ (and not from $\underline G-$) must depend on the height $h$ of this jump. In fact, it has the law of the original process conditioned to stay above $-h$. Thus splitting must be done at $\underline G-$ in this case, and the appropriate phase should be~$\underline J$.
Note, however, that we may avoid this problem by taking a MAP with all the phases in $E^\sim\cup E^\nearrow$, so that $\overline J=J_{\overline G}$ and $\underline J=J_{\underline G}$.

The matrix $\Delta_{\bs {\overline c}}^{-1}$ comes from time reversal of the post-infimum process. This reversal is not as straightforward as it may seem at first. The post-infimum process is not a MAP and hence the simple identity~\eqref{eq:timerev1} can not be applied. 
Assume anyway that there is an arbitrary diagonal matrix $\Delta$ in~\eqref{eq:WH} in place of $\Delta_{\bs {\overline c}}^{-1}$.
Right-multiply~\eqref{eq:WH} by the vector of ones and take $\a=0,\bb=\bs 0$ to arrive at
\[\bs 1=\p[\underline J]\Delta\hat\p[\overline J]^T\Delta_\pii\qq.\]
Using~\eqref{eq:K} observe that $\p[\underline J]$ is non-singular. But $\p[\underline J]\bs 1=\bs 1$ and so we must have $\Delta\hat\p[\overline J]^T\Delta_\pii\qq=\bs 1$ which shows that $\Delta$ must be $\Delta_{\bs {\overline c}}^{-1}$.
Finally, for the same killing rates $q_i=q$ the matrix $\Delta$ can be replaced by $q^{-1}\Delta_\pii^{-1}$ if $\pii\hat\p[\overline J]=\pii$, which is not true in general.
%
%

\section{Spectrally negative MAPs}\label{sec:spNeg}
An important special case of a MAP is obtained assuming that $X$ has no positive jumps and that none of $X^i$ is a non-increasing process. These assumptions imply that $X$ can pass over $x>0$ only by hitting this level, and it can do so in any phase with positive probability. Now the strong Markov property implies that $J_{\tau_x},x\geq 0$ is a Markov chain (starting in $i$ under $\p_i$, because $\tau_0=0$ a.s.).

There are three fundamental matrices associated to a spectrally negative MAP:
\begin{itemize}
\item $G$ - the transition rate matrix of the first passage Markov chain, i.e.\ it satisfies $\p[J_{\tau_x}]=e^{Gx}$ for all $x\geq 0$. It is known that $G$ is the right solution of a certain matrix integral equation, which compactly can be written as $\Psi(-G)=\matO$, see e.g.~\cite{ivanovs_G}.
\item $R$ - the left solution of the above matrix integral equation.
\item $H$ - the matrix of expected local times at~0, see~\cite{ivanovs_scale} for the precise definition.
\end{itemize}
All of these matrices can be computed from a given~$\Psi(\a)$ using e.g.\ a spectral method, see~\cite{ivanovs_G,ivanovs_risk}. Moreover, all three matrices are invertible and satisfy the identity
\[HR=GH.\]
Observe that $(X,J)$ under $\hat \p$ is also a spectrally negative MAP. Denoting the corresponding fundamental matrices by $\hat G,\hat R,\hat H$ we have the following relations
\begin{equation}\label{eq:rel_hat}\hat G=\Delta_\pii^{-1}R^T\Delta_\pii, \qquad \hat R=\Delta_\pii^{-1}G^T\Delta_\pii, \qquad \hat H=\Delta_\pii^{-1}H^T\Delta_\pii,\end{equation}
see~\cite{ivanovs_potential} for the latter facts.
Finally, we write $G^\bb,R^\bb,H^\bb$ for the fundamental matrices corresponding to $\Psi^\bb(\a)$, i.e.\ corresponding to the process additionally killed with rates $\bb$, see Section~\ref{sec:defective}.

\begin{remark}\label{rem:levy}
Consider the case when there is only one phase, i.e.\ $X$ is a \levy process. Then $G=R=-\Phi$, where $\Phi>0$ is the positive solution to $\Psi(\cdot)=0$. Moreover,~\cite[Prop.\ 4.1]{ivanovs_risk} shows that $H=1/\Psi'(\Phi)$. Recall that all these quantities depend on the implicit killing rate $q>0$. In particular, $\Phi(q)$ solves $\Psi(\cdot)-q=0$,
and $H=1/\Psi'(\Phi(q))=\Phi'(q)$.
\end{remark}

\subsection{General remarks}
In the following we adapt the results on splitting and time reversal to the case of a spectrally negative MAP.
Firstly, observe that all the phases must belong to $E^\sim\cup E^\nearrow$, because \levy processes of bounded variation with drift $d_i\leq 0$ are non-increasing (when there are no positive jumps) and thus excluded. Hence splitting at the infimum is always implemented at $\underline G$, i.e.\ $X_{\underline G}=\underline X$ and thus $\underline J=J_{\underline G}$ a.s., see Lemma~\ref{lem:underlineJ} and Proposition~\ref{prop:splitting}~(i). Similarly, in the case of supremum $\overline J=J_{\overline G-}$, i.e.\ splitting can always be implemented at $\overline G-$. Hence under $\p_i^\uparrow$ the process starts from $(0,i)$, but under $\p_i^\downarrow$ it may start from some $(x,j)$ where $x<0$, see Section~\ref{sec:positive}. 

Secondly, all $X^i$ admit continuous passage upward and $X_{\sigma_a}=a$ for $a\geq 0$, whenever $\sigma_a<\zeta$ which is the same as $J_{\sigma_a}=i$ for some phase~$i$. This makes the results on splitting and time reversal at $\sigma_a$ particularly simple.

\subsection{Wiener-Hopf factorization}\label{sec:spWH}
Let us provide the four factors involved in Corollary~\ref{cor:WH} in explicit form.
\begin{proposition}\label{prop:spneg}
There are the identities for $\qq >\bs 0,\bb\geq \bs 0$ and implicit killing rates $\qq\neq \bs 0$:
\begin{align}
\e[e^{\a \overline X-\langle\bb ,\overline {\bs G}\rangle};\overline J]&=(\a\matI+G^\bb)^{-1}\Delta_{(G\1)},\label{eq:WH_sup}\\
\e^\downarrow[e^{\a X_{\zeta-}-\langle\bb,\bs \zeta\rangle};J_{\zeta-}]&=-\Delta_{(G\1)}^{-1}(\a\matI+G^\bb)\left(\Psi^\bb(\a)\right)^{-1}\Delta_\qq\label{eq:WH_down},\\
\e[e^{\a \underline X-\langle\bb ,\underline {\bs G}\rangle};\underline J]&=
-\left(\Psi^\bb(\a)\right)^{-1}(\a\matI+R^\bb)\Delta_{(R^{-1}\qq)},\label{eq:WH_inf}\\
\e^\uparrow[e^{\a X_{\zeta-}-\langle\bb,\bs \zeta\rangle};J_{\zeta-}]&=\Delta^{-1}_{(R^{-1}\qq)}(\a\matI+R^\bb)^{-1}\Delta_\qq,\label{eq:WH_up}
\end{align}
where $\Re(\a)\leq 0$ in~\eqref{eq:WH_sup},~\eqref{eq:WH_up}, and $\Re(\a)\geq 0,\det(\Psi^\bb(\a))\neq 0$ in~\eqref{eq:WH_down},~\eqref{eq:WH_inf}.
\end{proposition}
\begin{proof}
Let $\bs g:=-G\1$ be the transition rates of $J_{\tau_x}$ into the absorbing state, i.e.\ the one corresponding to $\tau_x=\infty$.
Now we can write
\begin{align*}&\e[e^{\a \overline X-\langle\bb ,\overline {\bs G}\rangle};\overline J]=
\int_0^\infty \e[e^{\a x-\langle\bb,\bs {\tau_x}\rangle};J_{\tau_x}]\Delta_{\bs g}\D x=\int_0^\infty e^{\a x}\p^\bb[J_{\tau_x}]\D x\Delta_{\bs g}
\\&=\int_0^\infty e^{(\a\matI+G^\bb) x}\D x\Delta_{\bs g}
=-(\a\matI+G^\bb)^{-1}\Delta_{\bs g},\end{align*}
because all the eigenvalues of $G^\bb$ have negative real parts. This proves~\eqref{eq:WH_sup}, and then~\eqref{eq:WH_down} follows from Corollary~\ref{cor:WH} and analytic continuation. Here we also observe that $\bs g$ must have strictly positive entries, and $\p[\overline J]=-G^{-1}\Delta_{\bs g}$.

Next, from~\eqref{eq:WHup} and~\eqref{eq:rel_hat} we obtain
\begin{align*}\e^\uparrow[e^{\a X_{\zeta-}-\langle\bb,\bs \zeta\rangle};J_{\zeta-}]&=\Delta_{\bs {\overline c}}^{-1}\left(-(\a\matI+\hat G^\bb)^{-1}\Delta_{\bs {\hat g}}\right)^T\Delta_\pii\Delta_\qq
\\&=-\Delta_{\bs {\overline c}}^{-1}\Delta_{\bs {\hat g}}\Delta_\pii(\a\matI+R^\bb)^{-1}\Delta_\qq,\end{align*}
where 
\[-\bs {\overline c}=\left(\hat G^{-1}\Delta_{\bs {\hat g}}\right)^T\Delta_\pii\qq=\Delta_{\bs {\hat g}}\Delta_\pii R^{-1}\qq <\bs 0.
\]
This yields~\eqref{eq:WH_up}, and then~\eqref{eq:WH_inf} follows from Corollary~\ref{cor:WH} and analytic continuation. 
\end{proof}

In the following we provide some comments concerning Proposition~\ref{prop:spneg} and its relation to the known results in the literature.
Firstly, according to Remark~\ref{rem:levy} we may simplify~\eqref{eq:WH_sup} and~\eqref{eq:WH_inf} to
\begin{align*}
&\e(e^{\a \overline X-\b\overline G})=\frac{-\Phi}{\a-\Phi^\b},\\
&\e(e^{\a \underline X-\b\underline G})=\frac{1}{\Psi^\b(\a)}(\a-\Phi^\beta)\frac{q}{\Phi},
\end{align*}
where $\Phi,\Phi^\b,\Psi^\b(\a)$ become $\Phi(q),\Phi(q+\b),\Psi(\a)-q-\b$ when writing $q>0$ explicitly. These formulas are well known and can be found in e.g.~\cite[Sec.\ 6.5.2]{kyprianou}.

Secondly, Proposition~\ref{prop:spneg} can be extended to all $\bs q\geq \bs 0$ under certain assumptions. Let us consider the special case of a non-defective process, i.e.\ $\qq=\bs 0$. Letting 
\[\mu:=\e_\pii X_1\]
be the stationary drift of $X$, we note that~\eqref{eq:WH_sup} only makes sense when $\mu<0$, whereas~\eqref{eq:WH_inf} when $\mu>0$. The corresponding results are obtained by letting $q_i=q\downarrow 0$. 
Identity~\eqref{eq:WH_sup} has exactly the same form in the limit, but~\eqref{eq:WH_inf} becomes
\begin{equation}\label{eq:dieker}\e[e^{\a \underline X-\langle\bb ,\underline {\bs G}\rangle};\underline J]=
\left(\Psi^\bb(\a)\right)^{-1}(\a\matI+R^\bb)\mu\Delta_{\bs r}, \qquad \bs q=0,\mu>0,\end{equation}
where $\bs r$ is defined by the requirement
\begin{equation}\label{eq:r}
R\bs r=\bs 0, \qquad \pii\bs r=1,\qquad\qquad\text{ when }\qq=\bs 0,\mu>0.
\end{equation}
Indeed, such $\bs r$ is unique, non-negative and satisfies $-qR^{-1}\bs 1\rightarrow \mu\bs r$ as $q\downarrow 0$ (recall that $R$ depends on~$q$), which follows from~\eqref{eq:rel_hat} and~\cite[Lem.\ 3]{ivanovs_potential}.
It is noted that~\eqref{eq:dieker} can be found in~\cite{dieker_mandjes}, see Thm.\ 3.2., Lem.\ 3.1 and Lem.\ 3.2 therein, and recall that we exclude non-increasing processes. It may be interesting to note that~\cite{dieker_mandjes} is based on splitting for \levy processes leading to a certain embedding. This work does not state the Wiener-Hopf factorization for MAPs explicitly, but it does provide identities~\eqref{eq:WH_sup} and~\eqref{eq:WH_inf} for $\qq=0,\b_i=\b$. Importantly, the phases at $\overline G-$ and $\underline G$ are used in that work as well.

Thirdly, using splitting we may combine~\eqref{eq:WH_sup} and~\eqref{eq:WH_down} to obtain
\begin{align*}
\e[e^{\a X_{\zeta-}+\b\overline X};J_{\zeta-}]&=\e[e^{(\a+\b)\overline X};\overline J]\e^\downarrow[e^{\a X_{\zeta-}};J_{\zeta-}]\\
&=-((\a+\b)\matI+G)^{-1}(\a I+G)\Psi(\a)^{-1}\Delta_\qq.
\end{align*}
For $q_i=q>0$ this result coincides with~\cite[Cor.\ 4.21]{ivanovs_thesis}, where a different argument based on Asmussen-Kella martingale was employed. Similarly one can verify the result of~\cite[Thm.\ 10]{kyprianou_palm}.

\subsection{Last exit from the negative half-line}\label{sec:sn_last}
In this Section we identify the total times spent in different phases up to~$\sigma:=\sigma_0$ together with the phase at~$\sigma$. We provide three different approaches: first is based on splitting at $\sigma$ and law equivalence, see Corollary~\ref{cor:split_sigma}, second on the identity for time reversal at~$\sigma$, see~Proposition~\ref{prop:rev_last}, and third on splitting at the infimum and Theorem~\ref{thm:rev_first} providing the law of the post-infimum process up to its last passage. Generalization to $\sigma_a$ is trivial when $a>0$, and for $a<0$ it is possible but requires a so-called scale matrix. Finally, note that $J_{\sigma}=i$ implies that $\sigma<\zeta$ and hence $X_{\zeta-}>0$.


\begin{proposition}\label{prop:sigma}
For $\bb>0$ and $\qq> \bs 0$ it holds that
\[\e[e^{-\langle\bb,\bs \sigma\rangle};J_\sigma]=H^\bb\Delta_{(-R^{-1}\qq)}.\]
\end{proposition}
This result can be easily extended to the case of a non-defective process, i.e.\ $\qq=\bs 0$, assuming that $\mu=\e_\pii X_1>0$ which is the only non-trivial scenario. Letting $q_i=q\downarrow 0$ we immediately obtain
\[\e[e^{-\langle\bb,\bs \sigma\rangle};J_\sigma]=\mu H^\bb\Delta_{\bs r},\]
where $\bs r$ is defined in~\eqref{eq:r}.

In the case of a \levy process we may use Remark~\ref{rem:levy} to simplify the above identities:
\begin{align*}
\e^q(e^{-\b\sigma};\sigma<\zeta)&=\Phi'(\b+q)\frac{q}{\Phi(q)},\\
\e e^{-\b\sigma}&=\mu\Phi'(\b), \qquad \mu>0.
\end{align*}
The latter coincides with the result in~\cite{last_passage}; see also~\cite{ivanovs_SA} for a simple proof.

It should be noted that $\e[e^{-\langle\bb,\bs \sigma\rangle};J_\sigma]\neq \p^\bb[J_\sigma],$ because additional killing may result in a different~$\sigma<\zeta$. Recall that we do have equality of this type  for a deterministic $t$ and for $\tau_x$.

Finally, we remark that the form of the result of Proposition~\ref{prop:sigma}
is to be expected. Consider the case when all phases belong to $E^\nearrow$ and so there are finitely many visits to~0 a.s. It is easy to see that the resulting expression should be closely related to the expected number of these visits (in different phases under killing with rates $\bb$).
And indeed $H^\bb$ has this interpretation up to a certain scaling by a diagonal matrix from the right, see~\cite{ivanovs_scale}. This observation is made precise and complete in Section~\ref{sec:proof_rev}.

\subsubsection{Proof based on splitting at~$\sigma$}
The following proof heavily relies on splitting at $\sigma$, and the fact that post-$\sigma$ process has the law $\p^\uparrow$.
\begin{proof}[Proof of Proposition~\ref{prop:sigma}]
According to Corollary~\ref{cor:split_sigma} we have
\[\e[e^{-\langle\bb,\bs \sigma\rangle};J_{\sigma}]\p^{\uparrow}[e^{-\langle\bb,\bs \zeta\rangle};J_{\zeta-}]=\e[e^{-\langle\bb,\bs \zeta\rangle};X_{\zeta-}>0,J_{\zeta-}],\]
because $X_{\zeta-}>0$ if and only if $\sigma<\zeta$.
Next we compute
\begin{align*}
&\e[e^{-\langle\bb,\bs \zeta\rangle};X_{\zeta-}> 0,J_{\zeta-}]=
\int_0^\infty \e[e^{-\langle\bb,\bs t\rangle};X_{t}> 0,J_t]\D t\Delta_\qq\\
&=\int_0^\infty \p^\bb[X_{t}> 0,J_t]\D t\Delta_\qq
=\int_0^\infty H^\bb e^{R^\bb x}\D x\Delta_\qq=-H^\bb(R^\bb)^{-1}\Delta_\qq,
\end{align*}
where in the last line we used the expression for the potential density of $(X,J)$, see~\cite[Eq.\ (12)]{ivanovs_potential}.
Finally, according to Proposition~\ref{prop:spneg} we have
\[\e^\uparrow[e^{-\langle\bb,\bs \zeta\rangle};J_{\zeta-}]=\Delta^{-1}_{(R^{-1}\qq)}(R^\bb)^{-1}\Delta_\qq,\]
which completes the proof.
\end{proof}

\subsubsection{Proof based on time reversal at~$\sigma$}\label{sec:proof_rev}
The following proof relies on the identity of Proposition~\ref{prop:rev_last} and provides an explicit expression for~\eqref{eq:L} for a suitable functional~$F$.
\begin{proof}[Proof of Proposition~\ref{prop:sigma}]
Proposition~\ref{prop:rev_last} applied with $F=e^{-\langle\bs \beta,\bs \zeta\rangle}$ yields
\[\pi_iq_i\e_i(e^{-\langle\bs \beta,\bs \sigma\rangle};J_{\sigma}=j)=\hat u_j\hat L^0_{ji}(F),\]
because $\hat F_{\sigma}=F_{\sigma}=e^{-\langle\bs \beta,\bs \sigma\rangle}$.
Note that
\[\bs {\hat u}=\qq^T\Delta_\pii\int_0^\infty e^{\hat G x}\D x=-\qq^T\Delta_\pii\hat G^{-1}=-\qq^T (R^{-1})^T\Delta_\pii,\]
and 
\begin{align*}\hat L^0(F)&=\lim_{\epsilon\downarrow 0}\frac{1}{\epsilon}\hat \e[e^{-\langle\bs \beta,\bs \zeta\rangle};X_{\zeta-}\in(0,\epsilon),J_{\zeta-}]
=\lim_{\epsilon\downarrow 0}\frac{1}{\epsilon}\int_0^\infty\hat \p^\bb[X_{t}\in(0,\epsilon),J_{t}]\D t\Delta_\qq\\
&=\lim_{\epsilon\downarrow 0}\frac{1}{\epsilon}\int_0^{\epsilon}\hat H^{\bs \beta}e^{\hat R^{\bs \beta}x}\D x\Delta_\qq=\hat H^{\bs \beta}\Delta_\qq,\end{align*}
where we used the expression for the potential density again, see~\cite[Eq.\ (12)]{ivanovs_potential}. Combine these expressions and use~\eqref{eq:rel_hat} to complete the proof.
\end{proof}

\subsubsection{Proof based on splitting at the infimum}\label{sec:proof_split_inf}
Here we employ splitting at $\underline G$ and Theorem~\ref{thm:rev_first} expressing the law of the post-infimum process up to its last exit through the dual process reversed at its first passage. The analytic part of this proof is more complex and so we keep it rather brief. These ideas, but in a simpler form, were used in~\cite{ivanovs_SA} in the case of a \levy process.
\begin{proof}[Proof of Proposition~\ref{prop:sigma}]
Splitting at $\underline G$ yields
\[\e[e^{-\langle\bb,\bs \sigma\rangle};J_\sigma]=\int_{0-}^\infty\e[e^{-\langle\bs \beta,\bs {\underline G}\rangle};-\underline X\in \D x,\underline J]\e^\uparrow[e^{-\langle\bs \beta,\bs \sigma_x\rangle};J_{\sigma_x}].\]
According to Theorem~\ref{thm:rev_first} the $ji$th entry of the last matrix is given by
\[\frac{\hat u_i}{\hat u_j}\hat\e_i(e^{-\langle\bs \beta,\bs \tau_x\rangle};J_{\tau_x}=j)=\frac{\hat u_i}{\hat u_j}\left(e^{\hat G^\bb x}\right)_{ij}.\]
Using the expressions of $\hat {\bs u}$ in Section~\ref{sec:proof_rev} we obtain
\[\e^\uparrow[e^{-\langle\bs \beta,\bs \sigma_x\rangle};J_{\sigma_x}]=\Delta_{-R^{-1}\qq}^{-1}\Delta_\pii^{-1}e^{\hat {G^\bb}^T x}\Delta_\pii\Delta_{-R^{-1}\qq}=\Delta_{-R^{-1}\qq}^{-1}e^{R^\bb x}\Delta_{-R^{-1}\qq},\]
see also~\eqref{eq:rel_hat}.
Hence it is left to show that 
\[\int_{0-}^\infty\e[e^{-\langle\bs \beta,\bs {\underline G}\rangle};-\underline X\in \D x,\underline J]\Delta_{-R^{-1}\qq}^{-1}e^{R^\bb x}=H^\bb,\]
which can be achieved using the spectral method.
Take an eigenpair of $R$, i.e.~suppose that $R^\bb \bs v=-\lambda \bs v$, and multiply the left term in the above display by $\bs v$ from the right to get
\[\e[e^{-\langle\bs \beta,\bs {\underline G}\rangle+\lambda\underline X};\underline J]\Delta_{-R^{-1}\qq}^{-1}\bs v=\lim_{\ep\downarrow 0}\ep\left(\Psi^\bb(\lambda+\ep)\right)^{-1}\bs v\]
according to~\eqref{eq:WH_inf}; here we take $\ep\downarrow 0$, because $\Psi^\bb(\lambda)$ is singular. But this is exactly $H^\bb\bs v$ according to~\cite[Prop.\ 4.1]{ivanovs_risk}, which yields the result when the spectrum of $R$ is semi-simple. Finally, the general case can be treated as in~\cite{ivanovs_risk}.
\end{proof}

\subsection{The initial distribution of the conditioned process}\label{sec:initial_sn}
In this Section we study the initial distribution of $(X,J)$ under $\p_i^\downarrow$ for $i\in E^\nearrow$; recall that the absorbing state in this set-up is $(-\infty,\dagger)$. Alternatively, we may take the spectrally positive process $(-X,J)$ conditioned to stay positive. In our present setting these problems are the same since none of $X^j$ is a compound Poisson process and, moreover, the third scenario of Figure~\ref{fig:scenarios} can not hold, see Section~\ref{sec:non_pos}.

Thus we may apply Proposition~\ref{prop:init_law}, see also~\eqref{eq:Udef} for the definition of the jump measure $U(\D x)$.
First, observe that
\[\bs v(y):=\p[\overline X<y]=\bs 1-e^{G y}\bs 1,\quad y>0.\]
 Moreover, recall that $G$ is the right solution of a certain matrix integral equation, see e.g.~\cite{ivanovs_G}. In our setting this equation reads as
\begin{equation}\label{eq:FG}\int_{-\infty}^0 U_{i\cdot}(\D x)(e^{-G x}-\matI)+\Psi_{i\cdot}(0)-d_i\bs e_iG=\bs 0^T,\qquad i\in E^\nearrow,\end{equation}
where $d_i>0$ is the linear drift of $X^i$ and $\bs e_i$ is the $i$th unit vector. Multiply this equation by $-\bs 1$ from the right to find
\[c_i:=\int_{-\infty}^0 U_{i\cdot}(\D y)\bs v(-y)=-d_i\bs e_iG\bs 1-q_i.\]
Hence we have an explicit distribution of $(X_0,J_0)$ under $\p^\downarrow$:
\begin{align*}\p_i^\downarrow(X_0\in \D x;J_0=j)&=\frac{1}{-d_i\bs e_iG\bs 1}(1-\bs e_j e^{-G x}\bs 1)U_{ij}(\D x), \quad x<0.
\end{align*}

In the case of a spectrally negative \levy process of bounded variation and linear drift $d>0$ the above result simplifies to
\begin{align*}\p_i^\downarrow(X_0\in\D x;\zeta>0)=\frac{(1-e^{\Phi(q)x})\nu(\D x)}{d\Phi(q)}, \qquad x<0.\end{align*}
Suppose now that $\e X_1\geq 0$ so that $\Phi(0)=0$. Then the above expression has the following limit
\[\p_i^\downarrow(X_0\in \D x,\zeta>0)\rightarrow -x\nu(\D x)/d, \text{ as }q\downarrow 0, \qquad x<0.\]
This provides an alternative derivation of~\cite[Eq.~(2.12)]{chaumont_doney} in view of~\cite[Cor.\ 3.2]{bertoin_splitting}.


\end{document}